\documentclass[a4paper,11pt]{amsart}

\usepackage[all]{xy}
\usepackage{amsmath}
\usepackage{amssymb}
\usepackage{amsthm}
\usepackage{latexsym}
\usepackage{enumerate}
\usepackage{prettyref}
\usepackage{hyperref}
\usepackage{color}
\usepackage[margin=0.9in]{geometry}

\newtheorem{mainthm}{Theorem}
\newtheorem{conjecture}[mainthm]{Conjecture}
\newrefformat{conj}{\hyperref[{#1}]{Conjecture~\ref*{#1}}}
\newtheorem{question}[mainthm]{Question}
\newtheorem{theorem}{Theorem}[section]
\newrefformat{thm}{\hyperref[{#1}]{Theorem~\ref*{#1}}}

\newrefformat{def}{\hyperref[{#1}]{Definition~\ref*{#1}}}

\newrefformat{lem}{\hyperref[{#1}]{Lemma~\ref*{#1}}}
\newtheorem{proposition}[theorem]{Proposition}
\newrefformat{prop}{\hyperref[{#1}]{Proposition~\ref*{#1}}}
\newtheorem{corollary}[theorem]{Corollary}
\newrefformat{cor}{\hyperref[{#1}]{Corollary~\ref*{#1}}}
\newtheorem{remark}[theorem]{Remark}
\newrefformat{rem}{\hyperref[{#1}]{Remark~\ref*{#1}}}
{
\newtheorem{examplecore}[theorem]{Example}}
\newrefformat{ex}{\hyperref[{#1}]{Example~\ref*{#1}}}

\newcommand{\Pic}{\ensuremath{\operatorname{Pic}}}

\newcommand{\op}{\operatorname}

\begin{document}

\title{A refinement of a conjecture of Quillen}

\author{Alexander D. Rahm and Matthias Wendt}
\address{Alexander D. Rahm, Department of Mathematics, National
  University of Ireland at Galway}
\email{Alexander.Rahm@nuigalway.ie}

\address{Matthias Wendt, Fakult\"at Mathematik,
Universit\"at Duisburg-Essen, Thea-Leymann-Strasse 9, 
Essen, Germany} 
\email{matthias.wendt@uni-due.de}


\begin{abstract}
We present some new results on the cohomology of a large scope of SL$_2$-groups in degrees above the virtual cohomological dimension; 
yielding some partial positive results for the Quillen conjecture in rank one.
We combine these results with the known partial positive results and the known types of counterexamples to the Quillen conjecture,
in order to formulate a refined variant of the conjecture.
\end{abstract}
%
\maketitle

\section{Quillen's conjecture - formulation and history} 

In his fundamental work on the structure of equivariant cohomology
rings, cf. \cite{quillen:spectrum}, Quillen formulated a conjecture on
the structure of cohomology rings of certain $S$-arithmetic groups. In the
time that has passed since the formulation of the conjecture, it has
been proved in some cases and disproved in others, but the exact
nature of the conjecture and an explicit description of the cases
where it holds has not yet been found. Our goal in the present note is
to discuss some recent examples which shed new light on Quillen's
conjecture. Guided by these examples, we attempt a refined formulation
of the original conjecture. 

We first 
state Quillen's original conjecture, cf. \cite[Conjecture
14.7, p. 591]{quillen:spectrum}.  For any number field $K$, and any set
of places $S$ of $K$, the natural embedding
$\op{GL}_n(\mathcal{O}_{K,S})\hookrightarrow\op{GL}_n(\mathbb{C})$
induces a restriction map in cohomology
$$
\op{res}_{K,S}:
\op{H}^\bullet(\op{GL}_n(\mathbb{C}),\mathbb{F}_\ell)\to  
\op{H}^\bullet(\op{GL}_n(\mathcal{O}_{K,S}),\mathbb{F}_\ell). 
$$
Moreover, denoting by $c_i$ the $i$-th Chern class in
$\op{H}^\bullet_{\op{cts}}(\op{GL}_n(\mathbb{C}),\mathbb{F}_\ell)$,
there is a change-of-topology map 
$$
\delta:\mathbb{F}_\ell[c_1,\dots,c_n]\cong 
\op{H}^\bullet_{\op{cts}}(\op{GL}_n(\mathbb{C}),\mathbb{F}_\ell)\to 
\op{H}^\bullet(\op{GL}_n(\mathbb{C}),\mathbb{F}_\ell).
$$
The conjecture of Quillen can now be stated as follows: 
\begin{conjecture}[Quillen]
Let $\ell$ be a prime number. Let $K$ be a number field with
$\zeta_\ell\in K$, and $S$ a finite set of places containing the
infinite places and the places over $\ell$. Then the composition
$\op{res}_{K,S}\circ \delta$ makes
$\op{H}^\bullet(\op{GL}_n(\mathcal{O}_{K,S}),\mathbb{F}_\ell)$ a free
module over the cohomology ring
$\op{H}^\bullet_{\op{cts}}(\op{GL}_n(\mathbb{C}),\mathbb{F}_\ell)\cong
\mathbb{F}_\ell[c_1,\dots,c_n]$.  
\end{conjecture}
The range of validity of the conjecture has not yet been
decided. Positive cases in which the conjecture  has been established
are $n=\ell=2$ by Mitchell \cite{mitchell}, $n=3$, $\ell=2$ by Henn
\cite{henn}, and  $n=2$, $\ell=3$ by Anton \cite{anton}. 
Using \cite[Remark on p. 51]{henn:lannes:schwartz}, counterexamples to
Quillen's conjecture have been established by 
Dwyer \cite{dwyer}  for $n\geq 32$ and $\ell=2$, Henn and Lannes \cite{henn:lannes} for
$n\geq 14$ and $\ell=2$, and by Anton \cite{anton} for $n\geq 27$ and
$\ell=3$. 

Via the remark in \cite[p.51]{henn:lannes:schwartz}, the Quillen
conjecture has been viewed as closely related to the following
question, to which we will refer as ``detection question'' in the
sequel. 

\begin{question}[Detection question]
For which number fields $K$, place sets $S$ of $K$, primes $\ell$  and
natural numbers $n$ is the restriction morphism 
$
\op{H}^\bullet(\op{GL}_n(\mathcal{O}_{K,S}),\mathbb{F}_\ell)\to
\op{H}^\bullet(\op{T}_n(\mathcal{O}_{K,S}),\mathbb{F}_\ell)
$
injective, where $\op{T}_n$ is the group of diagonal matrices in
$\op{GL}_n$?  
\end{question}

However, while the remark in \cite{henn:lannes:schwartz} concerns only
the case $\op{GL}_n(\mathbb{Z}[1/2])$ with
$\mathbb{F}_2$-coefficients, the nature of the relation between
Quillen's conjecture and detection questions has not been made precise
yet. All we found in the literature was the following sentence on p. 13 of
\cite{knudson:book}:
``This [the Quillen conjecture] implies the following conjecture
[the detection question] in many cases.''
Unfortunately, the ``many cases'' are left unspecified. 
A secondary objective of the paper is to clarify the relation between
Quillen's conjecture and detection questions. 

\subsection{Acknowledgements:} We would like to thank Hans-Werner Henn
for an interesting email conversation concerning the remark in
\cite{henn:lannes:schwartz} and the Quillen conjecture. We would also like to thank the anonymous referee for helpful suggestions.

\section{Subgroup structure in high rank - negative results}
\label{sec:neg}

We first discuss the known counterexamples to the Quillen
conjecture. As mentioned above, these are built on a remark in
\cite{henn:lannes:schwartz} together with examples of the failure of
detection (due to non-triviality of class groups of group rings for
sufficiently complicated finite subgroups) as found by Dwyer
\cite{dwyer}, Henn--Lannes \cite{henn:lannes} and Anton
\cite{anton}. We provide a precise 
formulation of the result of Henn--Lannes--Schwartz. The proof given
below is mostly contained in \cite{henn:lannes:schwartz}, details
missing in loc.cit. were explained to us by Hans-Werner Henn - we
claim no originality except for mistakes we might have introduced.

\begin{proposition}[Henn--Lannes--Schwartz]
\label{prop:hls}
Let $\ell$ be a prime number. Let $K$ be a number field with
$\zeta_\ell\in K$, and $S$ a finite set of places containing the
infinite places and the places over $\ell$. 
Assume that all elementary abelian $\ell$-groups  in
$\op{GL}_n(\mathcal{O}_{K,S})$ are conjugate to subgroups of the
diagonal matrices, and that Quillen's 
conjecture holds for $K$, $S$ and $\ell$. Then detection holds for
$K$, $S$ and $\ell$.
\end{proposition}

\begin{proof}
We assume that detection does not hold, and want to derive a
contradiction. 

(i)  Let $E_0$ be the group of diagonal matrices of order $\ell$. This
is a maximal elementary abelian $\ell$-subgroup of
$\op{GL}_n(\mathcal{O}_{K,S})$. If detection fails, then the 
 restriction map
$$
\op{H}^\bullet(\op{GL}_n(\mathcal{O}_{K,S}),\mathbb{F}_\ell)\to
\op{H}^\bullet(E_0,\mathbb{F}_\ell)
$$
is not injective. This follows from functoriality of group cohomology,
because we have an inclusion
$E_0\leq\op{T}_n(\mathcal{O}_{K,S})\leq\op{GL}_n(\mathcal{O}_{K,S})$
and the restriction map associated to the second map is not injective
by assumption. 
For $g\in \op{GL}_n(\mathcal{O}_{K,S})$, the homomorphism $c\mapsto
gcg^{-1}:\op{GL}_n(\mathcal{O}_{K,S})\to\op{GL}_n(\mathcal{O}_{K,S})$
induces the identity on
$\op{H}^\bullet(\op{GL}_n(\mathcal{O}_{K,S}),\mathbb{F}_\ell)$,
c.f.e.g. \cite[Proposition A.1.11]{knudson:book}. Together with the
above argument, failure of detection implies that the following
product of restriction maps
is also not injective: 
\begin{equation} \label{(1)}
 \op{H}^\bullet(\op{GL}_n(\mathcal{O}_{K,S}),\mathbb{F}_\ell)\to 
\prod_{E\in\mathcal{M}}
\op{H}^\bullet(E,\mathbb{F}_\ell),
\end{equation}
where $\mathcal{M}$ is the set of all maximal elementary abelian
$\ell$-subgroups $E\leq\op{GL}_n(\mathcal{O}_{K,S})$.

(ii) Assume that there exists a class
$x\in\op{H}^\bullet(\op{GL}_n(\mathcal{O}_{K,S}),\mathbb{F}_\ell)$ which is
not a zero-divisor, and whose restriction to $E_0$ is not nilpotent
but in the essential ideal. Then \cite[Corollary
5.8]{henn:lannes:schwartz} implies that the above product of restriction maps (\ref{(1)})
is injective.
The result is applicable since the group
$\op{GL}_n(\mathcal{O}_{K,S})$ is of finite virtual 
cohomological dimension and the cohomology ring
$\op{H}^\bullet(\op{GL}_n(\mathcal{O}_{K,S}),\mathbb{F}_\ell)$ is
noetherian, cf. the discussion in \cite{quillen:spectrum}. 
Recall the collection $\mathcal{C}_x$ of \cite[Corollary
5.8]{henn:lannes:schwartz}: it is obtained as the collection of
elementary abelian $\ell$-subgroups $E$ of
$\op{GL}_n(\mathcal{O}_{K,S})$, such that the restriction res$_E(x)$
is not nilpotent.   
The collection $\mathcal{C}_x$ is equal to $\mathcal{M}$: 
by assumption, the restriction of $x$ to $E_0$ is not nilpotent, and
since all maximal elementary abelian $\ell$-subgroups are conjugate,
the same is true for all other $E\in\mathcal{M}$. On the other hand,
since $x$ is required to restrict to the essential ideal, its
restriction to every proper subgroup of $E_0$ is trivial, so the same
is true for all non-maximal elementary abelian $\ell$-subgroups. 

(iii) It now suffices to find an element
$x\in\op{H}_{\op{cts}}^\bullet(\op{GL}_n(\mathbb{C}),\mathbb{F}_\ell)$ whose
restriction to
$\op{H}^\bullet(\op{GL}_n(\mathcal{O}_{K,S}),\mathbb{F}_\ell)$ has the
properties required in (ii): failure of detection in (i) and the
assumption that $x$ is not a zero-divisor in (ii) contradict each
other. Therefore, $x$ has to be a zero-divisor and hence
$\op{H}^\bullet(\op{GL}_n(\mathcal{O}_{K,S}),\mathbb{F}_\ell)$ 
cannot be a
free $\op{H}_{\op{cts}}^\bullet(\op{GL}_n(\mathbb{C}),\mathbb{F}_\ell)$-module. 

(iv) The proof is completed by producing an element with the properties
in (iii). For the structure of the essential ideal in the cohomology
rings of elementary abelian $\ell$-groups, we refer to
\cite{aksu:green}. 
In the case $\ell=2$, the product of all non-zero classes in
$\op{H}^1(E_0,\mathbb{F}_2)$ is an essential non-zero-divisor; its
square is  induced from the product of all non-zero classes in
$\op{H}^2_{\op{cts}}(\op{GL}_n(\mathbb{C}),\mathbb{F}_2)$. In the case of odd 
$\ell$, the product of all  
non-zero classes in $\op{H}^2(E_0,\mathbb{F}_\ell)$ is an essential
non-zero-divisor which is Weyl-invariant and hence induced from
continuous cohomology of $\op{GL}_n(\mathbb{C})$. 
\end{proof}

\begin{remark}
This proposition justifies the application of \cite[p. 51]{henn:lannes:schwartz}
in \cite{anton}: all elementary abelian $3$-groups of maximal rank in
$\op{GL}_n(\mathbb{Z}[\zeta_3,1/3])$ are conjugate.
\end{remark}

\section{Subgroup structure in rank one - positive results}
\label{sec:onepos}

We next discuss the rank one case, i.e., the groups
$\op{SL}_2(\mathcal{O}_{K,S})$. In this case, the subgroup structure
(and consequently the cohomology above the virtual cohomological
dimension) is under good control. This allows to establish variants
and partial positive results related to the Quillen conjecture. 

\subsection{Quillen conjecture for Farrell-Tate cohomology}

The following is one of the main results of \cite{sl2ff}. It provides
a complete computation of Farrell-Tate cohomology for
$\op{SL}_2(\mathcal{O}_{K,S})$ based on explicit description of
conjugacy classes of finite cyclic subgroups and their normalizers in
$\op{SL}_2(\mathcal{O}_{K,S})$. Similar results can be established for
$\op{(P)GL}_2(\mathcal{O}_{K,S})$.

We first explain some notation. We will consider global fields $K$, place sets $S$ and primes $\ell$, and $\mathcal{O}_{K,S}$ denotes the relevant ring of $S$-integers. In the situation where $\zeta_\ell+\zeta_\ell^{-1}\in K$, we will abuse notation and write $\mathcal{O}_{K,S}[\zeta_\ell]$ to mean the ring $\mathcal{O}_{K,S}[T]/(T^2-(\zeta_\ell+\zeta_\ell^{-1})T+1)$. 
Moreover, we denote the norm maps for class groups and units by  
$$
\op{Nm}_0:
  \widetilde{\op{K}_0}(\mathcal{O}_{K,S}[\zeta_\ell])\to 
  \widetilde{\op{K}_0}(\mathcal{O}_{K,S})
\qquad\textrm{ and }\qquad
\op{Nm}_1:\mathcal{O}_{K,S}[\zeta_\ell]^\times\to
  \mathcal{O}_{K,S}^\times.
$$

\begin{theorem}
\label{thm:gl2nf}
Let $K$ be a global field, let $S$ be a non-empty finite set of places
of $K$ containing the infinite places, and let $\ell$ be an odd prime
different from the characteristic of $K$. 

\begin{enumerate}
\item
  $\widehat{\op{H}}^\bullet(\op{SL}_2(\mathcal{O}_{K,S}),\mathbb{F}_\ell)\neq
  0$ if and only if $\zeta_\ell+\zeta_\ell^{-1}\in K$ and the Steinitz
  class $\det_{\mathcal{O}_{K,S}}(\mathcal{O}_{K,S}[\zeta_\ell])$ is
  contained in the image of the norm map $\op{Nm}_0$.
\item Assume the condition in (1) is satisfied. 
  The set $\mathcal{C}_\ell$ of conjugacy classes of order $\ell$
  elements in $\op{SL}_2(\mathcal{O}_{K,S})$ 
  sits in an extension 
$$
1\to \op{coker}\op{Nm}_1\to \mathcal{C}_\ell\to 
\ker\op{Nm}_0\to 0.
$$
The set $\mathcal{K}_\ell$ of conjugacy classes of order $\ell$
subgroups of $\op{SL}_2(\mathcal{O}_{K,S})$ can be identified with the
quotient 
$\mathcal{K}_\ell=\mathcal{C}_\ell/\op{Gal}(K(\zeta_\ell)/K)$. There is
a direct sum decomposition  
$$
\widehat{\op{H}}^\bullet(\op{SL}_2(\mathcal{O}_{K,S}),\mathbb{F}_\ell)\cong  
\bigoplus_{[\Gamma]\in\mathcal{K}_\ell}
\widehat{\op{H}}^\bullet(N_{\op{SL}_2(\mathcal{O}_{K,S})}(\Gamma),
\mathbb{F}_\ell)
$$
which is compatible with the ring structure, i.e., the Farrell-Tate cohomology ring of $\op{SL}_2(\mathcal{O}_{K,S})$ is a direct sum of the sub-rings
for the subgroups $N_{\op{SL}_2(\mathcal{O}_{K,S})}(\Gamma)$.  

\item If the class of $\Gamma$ is not $\op{Gal}(K(\zeta_\ell)/K)$-invariant, 
then $N_{\op{SL}_2(\mathcal{O}_{K,S})}(\Gamma)\cong \ker\op{Nm}_1$.
There is a ring isomorphism
$$
\widehat{\op{H}}^\bullet(\ker\op{Nm}_1,\mathbb{Z})_{(\ell)}\cong
\mathbb{F}_\ell[a_2,a_2^{-1}]\otimes_{\mathbb{F}_\ell}\bigwedge
\left(\ker\op{Nm}_1\right).
$$ In particular, this is a free module
over the subring $\mathbb{F}_\ell[a_2^2,a_2^{-2}]$.
\item 
If the class of $\Gamma$ is $\op{Gal}(K(\zeta_\ell)/K)$-invariant, 
then there is an extension
$$
0\to \ker\op{Nm}_1\to N_{\op{SL}_2(\mathcal{O}_{K,S})}(\Gamma)\to \mathbb{Z}/2\to 1.
$$
There is a ring isomorphism
$$
\widehat{\op{H}}^\bullet(N_{\op{SL}_2(\mathcal{O}_{K,S})}(\Gamma),\mathbb{Z})_{(\ell)}\cong
\left(\mathbb{F}_\ell[a_2,a_2^{-1}]\otimes_{\mathbb{F}_\ell}\bigwedge
\left(\ker\op{Nm}_1\right)\right)^{\mathbb{Z}/2}, 
$$
with the $\mathbb{Z}/2$-action given by multiplication with $-1$ on
$a_2$ and $\ker\op{Nm}_1$. In particular, this is a free module over
the subring
$\mathbb{F}_\ell[a_2^2,a_2^{-2}]\cong
\widehat{\op{H}}^\bullet(D_{2\ell},\mathbb{Z})_{(\ell)}$. 
\item The restriction map induced from the inclusion
  $\op{SL}_2(\mathcal{O}_{K,S})\to
  \op{SL}_2(\mathbb{C})$ maps the second Chern class
  $c_2$ to the sum of the elements $a_2^2$ in all the components. 
\end{enumerate}
\end{theorem}

\begin{corollary}
Let $\ell$ be a prime number. Let $K$ be a number field with
$\zeta_\ell\in K$, and $S$ a finite set of places containing the
infinite places and the places over $\ell$. 
\begin{enumerate}
\item
The Quillen conjecture is true for Farrell-Tate cohomology of 
$\op{SL}_2(\mathcal{O}_{K,S})$. More precisely, the natural morphism  
$
\mathbb{F}_\ell[c_2]\cong
\op{H}^\bullet_{\op{cts}}(\op{SL}_2(\mathbb{C}),\mathbb{F}_\ell)\to
\op{H}^\bullet(\op{SL}_2(\mathcal{O}_{K,S}),\mathbb{F}_\ell)
$
extends to a morphism 
$$
\phi:\mathbb{F}_\ell[c_2,c_2^{-1}]\to
\widehat{\op{H}}^\bullet(\op{SL}_2(\mathcal{O}_{K,S}),\mathbb{F}_\ell)
$$
which makes 
$\widehat{\op{H}}^\bullet(\op{SL}_2(\mathcal{O}_{K,S}),\mathbb{F}_\ell)$
a free $\mathbb{F}_\ell[c_2,c_2^{-1}]$-module. 
\item The Quillen conjecture is true for
  $\op{SL}_2(\mathcal{O}_{K,S})$ in   cohomological degrees above the
  virtual cohomological dimension.  
\end{enumerate}
\end{corollary}

\begin{remark}
Above the virtual cohomological dimension, this is an
$\op{SL}_2$-analogue of the results of  \cite{anton}. 
\end{remark}

On the other hand, there are examples of the failure of detection for
$\op{SL}_2$. In particular, the Quillen conjecture does not generally
imply detection; some non-trivial hypothesis is necessary in
Proposition~\ref{prop:hls}. A rather simple example for the failure of
detection is given by $K=\mathbb{Q}(\zeta_{23})$, $S=\{(23)\}\cup
S_\infty$ and $\ell=23$. In this case, there are three conjugacy classes of
elements of order $23$ (corresponding to $\mathbb{Q}(\zeta_{23})$ having class number three) and two conjugacy classes of cyclic subgroups of order $23$ (the two non-trivial classes of elements forming one conjugacy orbit). Detection fails by a simple rank 
argument - the source of the restriction map has two copies of the cohomology of a dihedral extension of $\mathcal{O}_{K,S}^\times$, the target only one copy of the cohomology of $\mathcal{O}_{K,S}^\times$, cf. \cite{sl2ff}. A similar class of examples is  given by $K=\mathbb{Q}(\sqrt{-m},\zeta_3)$ with $m\equiv  1\mod 3$, $S=\{(3)\}\cup S_\infty$, $\ell=3$ for those (infinitely many) $m$ for which $K$ has class number $\geq 3$.  
These examples for the failure of detection apply to Farrell-Tate cohomology as well as to group cohomology above the virtual cohomological dimension. 

The computation of
$\widehat{\op{H}}^\bullet(\op{SL}_2(\mathcal{O}_{K,S}),\mathbb{F}_\ell)$
is obtained by considering the action of
$\op{SL}_2(\mathcal{O}_{K,S})$ on the associated symmetric space
$\mathfrak{X}_{K,S}$ (which is a product of hyperbolic upper half
spaces for the complex places, upper half planes for the real places,
and Bruhat-Tits trees for the finite places). It is possible to
describe completely the subspace of $\mathfrak{X}_{K,S}$ consisting of
points fixed by some finite subgroup. The local structure of this
subcomplex is determined by examining the representation theory of the
relevant finite groups on the ``tangent space'' of
$\mathfrak{X}_{K,S}$. The global structure only depends on
number-theoretic data: the connected components are in bijection with
conjugacy classes of finite cyclic subgroups, and the homotopy type of
each connected component is (up to the prime $2$) the classifying
space of the normalizer of the corresponding finite subgroup. The
conjugacy classification of finite cyclic subgroups and the
description of the normalizers is an extension of the classical
Latimer-MacDuffee theorem. After having this precise description, the
computation of the Farrell-Tate cohomology of
$\op{SL}_2(\mathcal{O}_{K,S})$ is immediate.

\subsection{Quillen conjecture in function field situations}

The Quillen conjecture can also be asked in function field
situations. Let $p$ and $\ell$ be distinct primes. 
By Quillen's computations, there is a natural element
$c_2\in \op{H}^4(\op{SL}_2(\overline{\mathbb{F}_p}),\mathbb{F}_\ell)$
such that we have an identification 
$\op{H}^\bullet(\op{SL}_2(\overline{\mathbb{F}_p}),\mathbb{F}_\ell)\cong
\mathbb{F}_\ell[c_2]$. This element comes from the roots of unity,
hence exists over any algebraically closed field of characteristic
$p$. In particular, there is a natural summand $\mathbb{F}_\ell[c_2]$
in $\op{H}^\bullet(\op{SL}_2(K),\mathbb{F}_\ell)$ for any
algebraically closed field $K$ of characteristic $p$. Note that
Friedlander's generalized isomorphism conjecture predicts that this
summand is the whole cohomology. 

It is then possible to ask if the natural map 
$$
\phi:\mathbb{F}_\ell[c_2]\to
\op{H}^\bullet(\op{SL}_2(k[C]),\mathbb{F}_\ell) 
$$
makes the cohomology ring a free module over the image of $\phi$, when
$k=\mathbb{F}_q$ such that $\ell\mid q-1$ or $k$ is an algebraically
closed field. The answer is similar to the number field case discussed
above, which follows from (a slight reformulation of) the results of
\cite{sl2parabolic}. 

\begin{theorem}
\label{thm:gl2ff}
Let $k=\mathbb{F}_q$ be a finite field, let $\ell$ be a
prime with $\ell\mid q-1$. Let $\overline{C}$ be a smooth projective
curve over $k$, let $P_1,\dots,P_s\in \overline{C}$ be closed points, and set
$C=\overline{C}\setminus\{P_1,\dots,P_s\}$.  
Then the parabolic cohomology (as defined in \cite{sl2parabolic}) has
a direct sum decomposition 
$\widehat{\op{H}}^\bullet(\op{SL}_2(k[C]),\mathbb{F}_\ell)\cong
  \bigoplus_{[\mathcal{L}]\in\mathcal{K}(C)}
\widehat{\op{H}}^\bullet(\Gamma_C(\mathcal{L}),\mathbb{F}_\ell)$, 
where the index set $\mathcal{K}(C)$ is the quotient set
  $\mathcal{K}(C)=\Pic(C)/\iota$ of the Picard   group of $C$ modulo
  the involution $\iota:\mathcal{L}\mapsto \mathcal{L}^{-1}$.  
The components of this direct sum are:
\begin{enumerate}
\item If 
  $\mathcal{L}|_C\not\cong\mathcal{L}|_C^{-1}$, then 
$
\widehat{\op{H}}^\bullet(\Gamma_C(\mathcal{L}),\mathbb{F}_\ell)\cong 
\op{H}^\bullet(k[C]^\times,\mathbb{F}_\ell).
$
\item If 
  $\mathcal{L}|_C\cong\mathcal{L}|_C^{-1}$, then 
$
\widehat{\op{H}}^\bullet(\Gamma_C(\mathcal{L}),\mathbb{F}_\ell)\cong 
\op{H}^\bullet(\widetilde{\mathcal{SN}},\mathbb{F}_\ell),
$
where $\widetilde{{\mathcal{SN}}}$ denotes the group of monomial matrices in
$\op{SL}_2(k[C])$.
\end{enumerate}
Since $\widehat{\op{H}}^i(\op{SL}_2(k[C]),\mathbb{F}_\ell)\cong
\op{H}^i(\op{SL}_2(k[C]),\mathbb{F}_\ell)$ 
for $i$ greater than the virtual $p'$-cohomological dimension of
$\op{SL}_2(k[C])$,  the above function field analogue of Quillen's
conjecture holds in those degrees.
\end{theorem}

The proof strategy is similar to the number field case: consider the
action of $\op{SL}_2(k[C])$ on the associated symmetric space (which
is a product of Bruhat-Tits trees for the places at infinity). It is
then possible to work out explicitly the structure of the subcomplex
of cells which are fixed by a non-unipotent non-central subgroup. The
components of this ``parabolic subcomplex'' are in bijection with a
quotient of the Picard group, and each component (up to the prime $2$)
has the homotopy type of the classifying space of its setwise
stabilizer. From this, again, the computation of the relevant
cohomology is immediate. 

Further explicit computations  exhibit function field
cases where Quillen's conjecture holds in all cohomological
degrees. 
The function field analogue of Quillen's
conjecture is true for $\op{SL}_2(\mathbb{F}_q[C])$ in the following
cases, which in some sense can be considered function field analogues
of the results of Mitchell \cite{mitchell} and Anton \cite{anton}:  
\begin{enumerate}
\item $C=\mathbb{P}^1\setminus\{\infty\}$ (Soul{\'e}),
\item $C=\mathbb{P}^1\setminus\{0,\infty\}$ and
  $\mathbb{P}^1\setminus\{0,1,\infty\}$ (\cite{sl2parabolic}, but see
  also \cite[Section 4.4]{knudson:book} and \cite{hutchinson})
\item $C=\overline{E}\setminus\{P\}$ with $\overline{E}$ an elliptic
  curve with a $k$-rational point $P$ (\cite[Section
  4.5]{knudson:book}). 
\end{enumerate}

\section{Non-detectable cohomology classes}
\label{sec:oneneg}

In the previous section, we have seen some positive results concerning
the Quillen conjecture in rank one, and we have seen that the results
for the number field and function field cases are close analogues. In
the function field case, however, it is possible to get new examples
of cases where the Quillen conjecture fails badly,
cf. \cite{sl2p1}. 

\begin{theorem}
Let $k=\mathbb{F}_q$ with $q\geq 11$, and let $\ell\nmid q$ be a prime. For
$C=\mathbb{P}^1\setminus\{0,1,\infty, P\}$ for some $k$-rational point
$P$, there exist cohomology classes in
$\op{H}^4(\op{GL}_2(\mathbb{F}_q[C]),\mathbb{F}_\ell)$ which 
cannot  be detected on any maximal torus or any finite subgroup. 
\end{theorem}

This result is proved by considering the action of $\op{GL}_2(k[C])$
on the associated building $\mathfrak{X}_C$ which is a product of four
Bruhat-Tits trees corresponding to the four points $0$, $1$, $\infty$
and $P$ on $\mathbb{P}^1$. The existence of many non-trivial cells in
the quotient $\op{GL}_2(k[C])\backslash\mathfrak{X}_C$ can basically
be traced to the fact that the configuration space of 4 points on
$\mathbb{P}^1$ is positive-dimensional. Similar results can be
obtained for $\mathbb{P}^1\setminus\{P_1,\dots,P_s\}$ with $s\geq 5$.

These counterexamples to the Quillen conjecture are of a different
nature than those discussed in Section~\ref{sec:neg} - they are not
related to finite subgroups, in fact they cannot  be detected on any
finite subgroup. This is a new obstruction to the Quillen conjecture,
which instead is (somehow) related to cusp forms. 

In the spirit of the analogy between number fields and function
fields, it makes sense to expect that the Quillen conjecture fails for
$\op{GL}_2(\mathbb{Z}[1/n])$ where $n$ has at least $3$ prime factors
(and hence the curve $\mathbb{Z}[1/n]$ has at least $4$ places ``at
infinity''). 

\section{Refinement of Quillen's conjecture}
\label{sec:refine}

With the results outlined in the previous sections, we now have some
more positive and negative cases of the Quillen conjecture at our
disposal. Assuming that all reasons for potential counterexamples
have been accounted for, we arrive at the 
following refinement of Quillen's conjecture. 

\begin{conjecture}
\label{conj:quillenref}
Let $K$ be a number field. Fix a prime $\ell$  such that
$\zeta_\ell\in K$, and an integer $n<\ell$. Assume that $S$ is a set
of places containing the infinite places and the places lying over
$\ell$.
If each cohomology class of $\op{GL}_n(\mathcal{O}_{K,S})$ is
detected on some finite subgroup, then
$\op{H}^\bullet(\op{GL}_n(\mathcal{O}_{K,S}),\mathbb{F}_\ell)$ is a 
free module over the image of the restriction map
$\op{H}^\bullet_{\op{cts}}(\op{GL}_n(\mathbb{C}),\mathbb{F}_\ell)\to 
\op{H}^\bullet(\op{GL}_n(\mathcal{O}_{K,S}),\mathbb{F}_\ell)$.
\end{conjecture}

We now discuss how the above refinement of Quillen's conjecture fits
in the landscape of known examples and counterexamples. 

\begin{enumerate}
\item 
Conjecture~\ref{conj:quillenref} is true for $\op{SL}_2$, as follows
from Theorem~\ref{thm:gl2nf} and Theorem~\ref{thm:gl2ff}. 
\item Requiring $\ell>n$ implies that $\ell$ does not divide the order
  of the Weyl group. All counterexamples of Section~\ref{sec:neg} are
  excluded by this requirement; the known counterexamples are for
  primes $2$ and $3$ in high enough rank. Generally, finite subgroups
  in Lie groups are  fairly complicated to handle. However, the special case of normalizers of elementary
  abelian subgroups for $\ell$ not dividing the order of the Weyl
  group is substantially simpler, cf. \cite{serre}; it is much closer
  to the rank one case of Section~\ref{sec:onepos}. One could hope that
  the groups appearing do not give rise to counterexamples coming from
  non-triviality of class groups of representation rings as in
  Section~\ref{sec:neg}. 
\item Requiring that all cohomology classes are detected on some
  finite subgroup excludes counterexamples of the type discussed in
  Section~\ref{sec:oneneg} (and allows the passage from $\op{SL}_2$ to
  $\op{GL}_2$ in Section~\ref{sec:onepos}). However, the results of
  Section~\ref{sec:oneneg} show that this requirement (at least in
  function field situations) is only rarely satisfied. 
\end{enumerate}

Finally, we should note that there is an implicit leap of faith in the
above conjecture lying in the passage from torsion-free modules to
free modules. Showing that the module is torsion-free is easier, we
only need to show that the second Chern class is not a
zero-divisor. The passage from torsion-free modules to free
modules is automatic in the case $\op{SL}_2$ because the cohomology
ring is a polynomial ring in one variable; but this may be much more
subtle in higher rank cases.
 
Certainly, the work done for the results in Section~\ref{sec:onepos}
shows the way how to investigate Conjecture~\ref{conj:quillenref},
cf. \cite{sl2ff}: away 
from the Weyl group, it is possible to work out the classification of
finite subgroups much more easily, cf. \cite{serre}. Then one can
consider the action of an $S$-arithmetic group $G(\mathcal{O}_{K,S})$
on the corresponding symmetric space. The structure of the
subcomplex fixed by finite-order elements can be understood locally in
terms of the representation of the finite subgroups on the tangent
spaces of their fixed points. Conjugacy classification of finite
subgroups in $S$-arithmetic groups can be reduced to number theory by
counting conjugacy classes in terms of ideal classes in suitable ring
extensions. The normalizers of finite subgroups of arithmetic groups
can be understood in terms of parabolic subgroups in algebraic
groups. The final hurdle is  the evaluation of the spectral sequence
and the description of the differentials. At least the case
$\op{SL}_3$ can still be understood, and will be investigated in a
forthcoming paper.

\end{document}